\documentclass[10pt, oneside,reqno]{amsart}
\usepackage{amsmath, amssymb, amsfonts}
\usepackage{amsthm}
\usepackage{hyperref}
\theoremstyle{definition}
\newtheorem{thm}[equation]{Theorem}

\newtheorem{prop}[equation]{Proposition}

\numberwithin{equation}{section}
\begin{document}
\author{Kamran Alam Khan}
\address{Department of Mathematics \\
 V. R. A. L. Govt. Girls P. G. College \\
 Bareilly (U.P.)-INDIA}
\email{kamran12341@yahoo.com}
\title{Generalized $n$-metric spaces and fixed point theorems}
\pagestyle{myheadings}
\thispagestyle{empty}
\begin{abstract}
\noindent
G\"{a}hler (\cite{GA1} ,\cite{GA2}) introduced the concept of  2-metric as a possible  generalization  of usual notion of a metric space. In many cases the results obtained in the usual metric spaces and 2-metric spaces are found to be unrelated (see \cite{HA1}). Mustafa and Sims \cite{MU2} took a different approach and introduced the notion of $G$-metric. The author \cite{KHA} generalized the notion of G-metric to more than three variables and introduced the concept of $K$-metric as a function $K\colon X^n \to \mathbb{R}^+$, $(n\ge 3)$. In this paper, We improve the definition of $K$-metric by making symmetry condition more general. This improved metric denoted by $G_n$ is called the \emph{Generalized $n$-metric}. We develop the theory for generalized $n$-metric spaces and obtain some fixed point theorems.
\end{abstract}
\maketitle
\section{INTRODUCTION}
G\"{a}hler (\cite{GA1} ,\cite{GA2}) introduced the concept of  2-metric as a possible  generalization  of usual notion of a metric space. K. S. Ha et al \cite{HA1} have pointed out that the construction by G\"{a}hler is an independent approach and in many cases there is no connection between the results obtained in the usual metric spaces and 2-metric spaces. It was mentioned by G\"{a}hler \cite{GA1} that the notion of a 2-metric is an extension of an idea of ordinary metric and geometrically $d(x,y,z)$ represents the area of a triangle formed by the points $x$,$y$ and $z$ in $X$ as its vertices. But this is not always true. Sharma \cite{SHA} showed that $d(x,y,z)=0$ for any three distinct points $x$, $y$, $z$ $\in \mathbb{R}^2$.\\  
\newline B. C. Dhage \cite{DHA} introduced the concept of $D$-metric in order to translate results from usual metric space to $D$-metric space. Mustafa and Sims \cite{MU1} showed that most of the results concerning $D$-metrics are incorrect. This led them to introduce a new class of generalized metrics called $G$-metric in which the tetrahedral inequality is replaced by an inequality involving repetition of indices (see \cite{MU2}).Many authors(such as \cite{ASK}) obtained fixed point results for $G$-metric spaces. Recently the author \cite{KHA} generalized the notion of $G$-metric space to more than three variables and introduced the concept of $K$-metric. In the present work we improve the definition of $K$-metric by making symmetry condition more general. This improved metric denoted by $G_n$ is called the \emph{Generalized $n$-metric}. We develop the theory for generalized $n$-metric spaces and obtain some fixed point theorems.\\

\subsection{Definition} (\cite{KHA}) Let $X$ be a non-empty set, and $\mathbb{R}^+$ denote the set of non-negative real numbers. Let $K\colon X^n \to \mathbb{R}^+$, $(n\ge 3)$ be a function satisfying the following properties:
\begin{itemize}
\item [[K 1]] $K(x_1,x_2,...,x_n)=0$ if $x_1=x_2=\dots =x_n$,
\item [[K 2]] $K(x_1,x_1,...,x_1,x_2)>0$ for all $x_1$, $x_2 \in X$ with $x_1\neq x_2$,
\item [[K 3]] $K(x_1,x_1,...,x_1,x_2)\leq K(x_1,x_2,...,x_n)$ for all $x_1$, $x_2,... ,x_n \in X$ with the condition that any two of the points $x_2,\cdots ,x_n$ are distinct,
\item [[K 4]] $K(x_1,x_2,...,x_n)=K(x_{\pi^r(1)},x_{\pi^r (2)},...,x_{\pi^r(n)})$, for all $x_1$, $x_2,...,x_n \in X$ and a permutation $\pi$ of $\{1,2,...n\}$ such that $\pi(s)=s+1$ for all\\ $1\le s<n$, $\pi(n)=1$ and for all $r\in \mathbb{N}$, 
\item [[K 5]] $K(x_1,x_2,...,x_n)\le K(x_1,x_{n+1},...,x_{n+1})+K(x_{n+1},x_2,...,x_n)$ for all\\ $x_1$,$x_2,...,x_n,x_{n+1}\in X$.
 
\end{itemize}
Then the function $K$ is called a \emph{$K$-metric} on $X$, and the pair $(X,K)$ a $ \emph{$K$-metric space}$. 

\subsubsection{Example} Let $\mathbb{R}$ denote the set of all real numbers. Define a function\\ $\rho\colon \mathbb{R}^n \to \mathbb{R}^+$,$(n\ge3)$ by
\begin{equation*}
\rho(x_1,x_2,\dots ,x_n)= \text{max} \{ \left|x_1-x_2\right|,\dots ,\left|x_{n-1}-x_n\right|,\left|x_n-x_1\right|\}
\end{equation*}
for all $x_1$, $x_2,...,x_n \in \mathbb{R}$. Then $( \mathbb{R}, \rho)$ is a $K$-metric space.

\subsubsection{Example} For any metric space $(X,d)$, the following functions define $K$-metrics on $X$:
\begin{itemize}
\item [(1)] $K_1^d(x_1,x_2,...,x_n)=\frac{1}{n}\big[\sum_{r=1}^{n-1}d(x_r,x_{r+1})+d(x_n,x_1)\big]$,

\item [(2)] $K_2^d(x_1,x_2,...,x_n)=\text{max}\{d(x_1,x_2),d(x_2,x_3),\dots ,d(x_{n-1},x_n),d(x_n,x_1)\}$.
\end{itemize}
Geometrically the $K$-metric represents the notion of the perimeter of an oriented polygon with vertices $x_1$,$x_2$,...,$x_n$. Here we observe that the condition of symmetry is not satisfied in general as in $G$-metric.  Thus the notion of a $K$-metric is not a straight forward translation of the concept of $G$-metric. Now we introduce an $n$ point analogue of $G$-metric as follows.

\section{main results}

\subsection{Definition}
Let $X$ be a non-empty set, and $\mathbb{R}^+$ denote the set of non-negative real numbers. Let $G_n\colon X^n \to \mathbb{R}^+$, $(n\ge 3)$ be a function satisfying the following properties:
\begin{itemize}
\item [[G 1]] $G_n(x_1,x_2,...,x_n)=0$ if $x_1=x_2=\dots =x_n$,
\item [[G 2]] $G_n(x_1,x_1,...,x_1,x_2)>0$ for all $x_1$, $x_2 \in X$ with $x_1\neq x_2$,
\item [[G 3]] $G_n(x_1,x_1,...,x_1,x_2)\leq G_n(x_1,x_2,...,x_n)$ for all $x_1$, $x_2,... ,x_n \in X$ with the condition that any two of the points $x_2,\cdots ,x_n$ are distinct,
\item [[G 4]] $G_n(x_1,x_2,...,x_n)=G_n(x_{\pi(1)},x_{\pi (2)},...,x_{\pi(n)})$, for all $x_1$, $x_2,...,x_n \in X$ and every permutation $\pi$ of $\{1,2,...n\}$,  
\item [[G 5]] $G_n(x_1,x_2,...,x_n)\le G_n(x_1,x_{n+1},...,x_{n+1})+G_n(x_{n+1},x_2,...,x_n)$ for all\\ $x_1$,$x_2,...,x_n,x_{n+1}\in X$. 
 
\end{itemize}
Then the function $G_n$ is called a \emph{Generalized $n$-metric} on $X$, and the pair $(X,G_n)$ a $ \emph{Generalized $n$-metric space}$.\\ 
From now on we always have $n\ge 3$ for $(X,G_n)$ to be a generalized $n$-metric space.

\subsubsection{Example} Define a function $\rho\colon \mathbb{R}^n \to \mathbb{R}^+$,$(n\ge3)$ by
\begin{equation*}
\rho(x_1,x_2,\dots ,x_n)= \text{max} \{ \left|x_r-x_s\right|\colon r,s\in \{1,2,...n\},r\neq s \}
\end{equation*}
for all $x_1$, $x_2,...,x_n \in X$. Then $( \mathbb{R}, \rho)$ is a generalized $n$-metric space.

\subsubsection{Example}
\label{exmpl} For any metric space $(X,d)$, the following functions define generalized $n$-metrics on $X$:
\begin{itemize}
\item [(1)] $K_1^d(x_1,x_2,...,x_n)=\sum_{r} \sum_{s}d(x_r,x_s)$,

\item [(2)] $K_2^d(x_1,x_2,...,x_n)=\text{max}\{d(x_r,x_s)\colon r,s\in \{1,2,...,n\},r\neq s\}$.
\end{itemize}

\begin{prop}
\label{inequalty}
Let $G_n\colon X^n \to \mathbb {R}^+$,$(n\ge 3)$ be a generalized $n$-metric defined on $X$, then for $x$,$y$ $\in X$ we have\\
\begin{equation}
\label{ineq1}
G_n(x,y,y,\dots ,y)\leq (n-1) G_n(y,x,x,\dots ,x)
\end{equation}
\end{prop}
\begin{proof} Using [G 5] it is trivial to prove the result.
\end{proof}
\subsection{Definition} Let $(X,G_n)$ be a generalized $n$-metric space, then for $x_0\in X$,$r>0$, the \emph{$G_n$-ball} with centre $x_0$ and radius $r$ is
\begin{equation*}
B_G(x_0,r)=\{ y\in X \colon G_n(x_0,y,y,\dots ,y)<r \}
\end{equation*}
\begin{prop}Let $(X,G_n)$ be a generalized $n$-metric space, then the $G_n$-ball is open in $X$.
\end{prop}
\begin{proof} The proof is straightforward.
\end{proof}

Hence the collection of all such balls in $X$ is closed under arbitrary union and finite intersection and therefore induces a topology on $X$ called the generalized $n$-metric topology $\Im (G_n)$ generated by the generalized $n$-metric on $X$.\\
From example~\ref{exmpl} it is clear that for a given metric we can always define generalized $n$-metrics. The converse is also true for if $G_n$ is a generalized $n$-metric then we can define a metric $d_G$ as follows-
\begin{equation*}
d_G(x,y)=G_n(x,y,y,\dots ,y)+G_n(x,x,\dots ,x,y)
\end{equation*}

\begin{prop} Let $B_{d_G}(x,r)$ denote the open ball in the metric space $(X,d_G)$ and $B_G(x,r)$ the $G_n$-ball in the correponding generalized $n$-metric space $(X,G_n)$. Then we have
\begin{equation*}
B_G(x,\frac{r}{n})\subseteq B_{d_G}(x,r)
\end{equation*}
\end{prop}
\begin{proof} Let $y\in B_G(x,\frac{r}{n})$ then $G_n(x,y,y,\dots ,y)<\frac{r}{n}$. From ~\eqref{ineq1} and [G 4] we have 
\begin{equation*}
G_n(x,x,\dots ,x,y) \le (n-1)G_n(x,y,y,...,y)<(n-1)r/n
\end{equation*}
Therefore
\begin{multline*} d_G(x,y)=G_n(x,y,y,\dots ,y)+G_n(x,x,\dots ,x,y) < \frac{r}{n}+(n-1)\frac{r}{n}=r \\
\end{multline*}
Henc we have $y\in B_{d_G}(x,r)$ and therefore $B_G(x,\frac{r}{n})\subseteq B_{d_G}(x,r)$
\end{proof}

This indicates that the topology induced by the generalized $n$-metric on $X$ coincides with the metric topology induced by the metric $d_G$. Thus every generalized $n$-metric space is topologically equivalent to a metric space.

\subsection{Definition} Let $(X,G_n)$ be a generalized $n$-metric space. A sequence $<x_m>$ in $X$ is said to be \emph{$G_n$-convergent} if it converges to a point $x$ in the generalized $n$-metric topology $\Im (G_n)$ generated by the $G_n$-metric on $X$.

\begin{prop}
\label{convsq} Let $G_r\colon X^r \to \mathbb{R}^+$, $(r\ge 3)$ be a generalized $r$-metric defined on $X$. Then for a sequence $<x_n>$ in $X$ and $x\in X$ the following are equivalent:
\begin{itemize}
\item [(1)] The sequence $<x_n>$ is $G_r$-convergent to $x$.
\item [(2)] $d_G(x_n,x)\rightarrow 0$ as $n\rightarrow \infty$.
\item [(3)] $G_r(x_n,x_n,...,x_n,x)\rightarrow 0$ as $n\rightarrow \infty$.
\item [(4)] $G_r(x_n,x,...,x)\rightarrow 0$ as $n\rightarrow \infty$.
\end{itemize}
\end{prop}
\begin{proof} Since the topology induced by the $G_r$-metric on $X$ coincides with the metric topology induced by the metric $d_G$, hence (1)$\Leftrightarrow$(2).\\
Now 
\begin{equation}
\label{equ1}
d_G(x_n,x)=G_r(x_n,x,\dots ,x)+G_r(x_n,x_n,\dots ,x_n,x)
\end{equation}
Hence $G_r(x_n,x,...,x)\rightarrow 0$ and $G_r(x_n,x_n,...,x_n,x)\rightarrow 0$ whenever $d_G(x_n,x)\rightarrow 0$. Thus (2)$\Rightarrow$(3) and (2)$\Rightarrow$(4).\\
From ~\eqref{ineq1} we have
\begin{equation}
\label{equ2}
G_r(x_n,x,...,x)\le (r-1)G_r(x,x_n,...,x_n)
\end{equation}
Thus (3)$\Rightarrow$(4). Similarly (4)$\Rightarrow$(3).\\
Also from ~\eqref{equ1} and ~\eqref{equ2} we have 
\begin{equation*}
d_G(x_n,x)\le rG_r(x_n,x_n,...,x_n,x)
\end{equation*}
Therefore (3)$\Rightarrow$(2).
\end{proof}

\subsection{Definition}
Let $(X,G^X_n)$ and $(Y,G^Y_n)$ be generalized $n$-metric spaces. A function $f\colon X \to Y$ is said to be \emph{Generalized $n$-continuous} at a point $x\in X$ if $f^{-1}(B_{G^Y_n}(f(x),r))\in \Im(G^X_n)$, for all $r>0$. The function $f$ is said to be generalized $n$-continuous if it is generalized $n$-continuous at all points of $X$.\\

Since every generalized $n$-metric space is topologically equivalent to a metric space, hence we have the following result:
\begin{prop}
\label{cont}
 Let $(X,G^X_n)$ and $(Y,G^Y_n)$ be generalized $n$-metric spaces. A function $f\colon X \to Y$ is said to be generalized $n$-continuous at a point $x\in X$ if and only if it is generalized $n$-sequentially continuous at $x$; that is, whenever the sequence $<x_m>$ is $G^X_n$-convergent to $x$, the sequence $<f(x_m)>$ is $G^Y_n$-convergent to $f(x)$.
\end{prop}

\begin{prop} Let $(X,G_n)$ be a generalized $n$-metric space, then the function $G_n(x_1,x_2,...,x_n)$ is jointly continuous in the variables $x_1$,$x_2,...,x_n$.
\end{prop}
\begin{proof} Let $<x_{m_1}>$,$<x_{m_2}>$,...,$<x_{m_n}>$ be the sequences in the generalized $n$-metric space $(X,G_n)$ such that $x_{m_1}\rightarrow x_1$, $x_{m_2}\rightarrow x_2$,...,$x_{m_n}\rightarrow x_n$.
Then by [G 4] and [G 5] we can show that
\begin{align*}
G_n(x_{m_1},x_{m_2},...,x_{m_n})-G_n(x_1,x_2,...,x_n)&\le G_n(x_{m_1},x_1,...,x_1)+G_n(x_{m_2},x_2,...,x_2)\\
&\quad +\dots +G_n(x_{m_n},x_n,...,x_n)
\end{align*}
Similarly
\begin{align*}
G_n(x_1,x_2,...,x_n)-G_n(x_{m_1},x_{m_2},...,x_{m_n})&\le G_n(x_1,x_{m_1},...,x_{m_1})+G_n(x_2,x_{m_2},...,x_{m_2})\\
&\quad +\dots +G_n(x_n,x_{m_n},...,x_{m_n})
\end{align*}
Therefore on using ~\eqref{ineq1} we have
\begin{align*}
\left|G_n(x_{m_1},x_{m_2},...,x_{m_n})-G_n(x_1,x_2,...,x_n)\right| \le &(n-1)\{G_n(x_1,x_{m_1},...,x_{m_1})\\
&+G_n(x_2,x_{m_2},...,x_{m_2})+\dots +G_n(x_n,x_{m_n},...,x_{m_n})\}
\end{align*}
Making $m_1\rightarrow \infty$, $m_2\rightarrow \infty$,...,$m_n\rightarrow \infty$ we have
\begin{equation*}
G_n(x_{m_1},x_{m_2},...,x_{m_n})\rightarrow G_n(x_1,x_2,...,x_n)
\end{equation*}
Hence the result follows. 
\end{proof}

\subsection{Definition}
\label{cauchy} Let $(X,G_m)$ be a generalized $m$-metric space. A sequence $<x_n>$ in $X$ is said to be $G_m$-Cauchy if for every $\epsilon >0$, there exists $N\in \mathbb{N}$ such that
\begin{equation*}
G_m(x_{n_1},x_{n_2},...,x_{n_m})< \epsilon \; \text{for all} \; n_1,n_2,...,n_m\ge N
\end{equation*}
\begin{prop}
Let $(X,G_m)$ be a generalized $m$-metric space. A sequence $<x_n>$ in $X$ is $G_m$-Cauchy if and only if for every $\epsilon >0$, there exists $N\in \mathbb{N}$ such that
\begin{equation}
\label{condcchy}
G_m(x_{n_1},x_{n_2},...,x_{n_2}) <\epsilon \; \text{for all} \; n_1,n_2\ge N
\end{equation}
\end{prop}
\begin{proof} If $<x_n>$ is $G_m$-Cauchy then the result follows from definition~\ref{cauchy}.\\
Conversely suppose that the condition~\eqref{condcchy} holds for a sequence $<x_n>$ in $X$. Then for $n_1$,$n_2$,$n_3\ge N$ we have from [G 5]
\begin{align*}
\begin{split}
G_m(x_{n_1},x_{n_2},x_{n_3},...,x_{n_3})&\le G_m(x_{n_1},x_{n_3},...,x_{n_3})+G_m(x_{n_3},x_{n_2},x_{n_3},...,x_{n_3})\\
&<\epsilon+\epsilon=2\epsilon
\end{split}
\end{align*}
Continuing the above argument, for $n_1$,$n_2$,...,$n_m\ge N$ we have
\begin{equation*}
G_m(x_{n_1},x_{n_2},\dots ,x_{n_m})<(m-1)\epsilon
\end{equation*}
i.e. $<x_n>$ is $G_m$-Cauchy.
\end{proof}

\begin{prop}
Every $G_n$-convergent sequence in a generalized $n$-metric space is $G_n$-Cauchy.
\end{prop}
\begin{proof} The result follows from proposition~\ref{convsq} and ~\eqref{condcchy}.
\end{proof}

\subsection{Definition} A generalized $n$-metric space $(X,G_n)$ is said to be \emph{$G_n$-complete} if every $G_n$-Cauchy sequence in $(X,G_n)$ is $G_n$-convergent in $(X,G_n)$.
\begin{thm} Let $G_r\colon X^r\to \mathbb{R}^+$,$(r\ge 3)$ be a generalized $r$-metric and $(X,G_r)$ be a $G_r$-complete generalized $r$-metric space. Let $f$ and $g$ be self mappings on $X$ satisfying the following conditions:
\begin{itemize}
\item [(1)] $f(X)\subseteq g(X)$,
\item [(2)] $g$ is continuous,
\item [(3)] $G_r(f\xi_1,f\xi_2,\dots ,f\xi_r)\le q\, G_r(g\xi_1,g\xi_2,\dots,g\xi_r)$ for every $\xi_1,\xi_2,\dots,\xi_r\in X$ and $0<q<1$
\end{itemize}
Then $f$ and $g$ have a unique common fixed point in $X$ provided $f$ and $g$ commute.
\end{thm}
\begin{proof} Let $x_0$ be an arbitrary point in $X$. Since $f(X)\subseteq g(X)$ hence there exists a point $x_1$ such that $fx_0=gx_1$. In general we can choose $x_{n+1}$ such that $y_n=fx_n=gx_{n+1}$.Using (3) we have
\begin{equation*}
G_r(fx_n,fx_{n+1},\dots,fx_{n+1})\le q\, G_r(gx_n,gx_{n+1},\dots ,gx_{n+1})\\
=q\,G_r(fx_{n-1},fx_n,\dots ,fx_n)
\end{equation*}
Proceeding in above manner we have
\begin{align*}
\begin{split}
G_r(fx_n,fx_{n+1},\dots,fx_{n+1})&\le q^n\, G_r(fx_0,fx_1,\dots ,fx_1)\\
\Rightarrow G_r(y_n,y_{n+1},\dots,y_{n+1})&\le q^n\,G_r(y_0,y_1, \dots,y_1)
\end{split}
\end{align*}
We claim that the sequence $<y_n>$ in $X$ is $G_r$-Cauchy in $X$.\\
For all natural numbers $n$ and $m(>n)$ we have from [G 5]
\begin{align*}
\begin{split}
G_r(y_n,y_m,\dots,y_m)&\le G_r(y_n,y_{n+1},\dots,y_{n+1})+G_r(y_{n+1},y_{n+2},\dots,y_{n+2})+\dots\\
&\qquad \cdots +G_r(y_{m-1},y_m,\dots ,y_m)\\
&\le (q^n+q^{n+1}+\dots +q^{m-1})\,G_r(y_0,y_1,\dots ,y_1)\\
&\le (q^n+q^{n+1}+\dots)\,G_r(y_0,y_1,\dots ,y_1)\\
&=\frac{q^n}{1-q}\,G_r(y_0,y_1,\dots ,y_1)\rightarrow 0 \, \text{as} \, n,m \rightarrow \infty
\end{split}
\end{align*}
Thus the sequence $<y_n>$ is a $G_r$-Cauchy sequence in $X$. By completeness of $(X,G_r)$, there exists a point $u\in X$ such that $<y_n>$ is $G_r$-convergent to $u$.Since $y_n=fx_n=gx_{n+1}$ hence we have $\lim_{n\to \infty}y_n=\lim_{n\to \infty}gx_n=\lim_{n\to \infty}fx_n=u$.\\
Now $g$ is continuous hence
\begin{equation*}
\lim_{n\to \infty}ggx_n=\lim_{n\to \infty}gfx_n=gu
\end{equation*}
Also $f$ and $g$ commute, therefore
\begin{equation*}
\lim_{n\to \infty}fgx_n=\lim_{n\to \infty}gfx_n=\lim_{n\to \infty}ggx_n=gu
\end{equation*}
Taking $\xi_1=gx_n,\, \xi_k=x_n\,(2\le k\le r)$ in (3) we have 
\begin{equation*}
G_r(fgx_n,fx_n,\dots ,fx_n)\le q\,G_r(ggx_n,gx_n,\dots ,gx_n)
\end{equation*}
Making $n\to \infty$ we have
\begin{equation*}
G_r(gu,u,\dots ,u)\le q\, G_r(gu,u,\dots ,u)
\end{equation*}
Which gives $gu=u$. For otherwise $q\ge 1$ contradicting the fact that $0<q<1$.
Now by taking $\xi_1=x_n$, $\xi_k=u\, (2\le k\le r)$ in (3) we have
\begin{equation*}
G_r(fx_n,fu,\dots ,fu)\le q\,G_r(gx_n,gu,\dots ,gu)
\end{equation*}
Making $n\to \infty$ we have $fu=u$. Therefore we have $fu=gu=u$, i.e. $u$ is a common fixed point of $f$ and $g$.\\
For uniqueness of $u$, suppose that $v\neq u$ is such that $fv=gu=v$. Then we have $G_r(u,v,\dots ,v)>0$ and
\begin{align*}
\begin{split}
G_r(u,v,\dots ,v)&=G_r(fu,fv, \dots ,fv)\le q\,G_r(gu,gv,\dots ,gv)=q\, G_r(u,v,\dots ,v)\\
&<G_r(u,v,\dots ,v)
\end{split}
\end{align*} 
Thus we get a contradiction, hence we have $u=v$.
\end{proof}
\begin{thm} Let $G_r\colon X^r\to \mathbb{R}^+$,$(r\ge 3)$ be a generalized $r$-metric and $(X,G_r)$ be a $G_r$-complete generalized $r$-metric space. Let $f\colon X\to X$ be a mapping which satisfies the following condition for all $x_1,x_2,\dots ,x_r\in X$
\begin{align*}
\begin{split}
G_r(fx_1,fx_2, \dots, fx_r) \le k\,\text{max} &\{ G_r(x_1,x_2,\dots, x_r),G_r(x_1,fx_1,\dots ,fx_1),\dots , G_r(x_r,fx_r,\dots ,fx_r),\\
&G_r(x_1,fx_2,\dots ,fx_2), G_r(x_2,fx_3,\dots ,fx_3),\dots ,G_r(x_r,fx_1,\dots ,fx_1)\}
\end{split}
\end{align*}
Where $0\le k<1/2$. Then $f$ has a unique fixed point (say $u$) and $f$ is generalized $r$-continuous at $u$.
\end{thm}
\begin{proof}
Let $f\colon X\to X$ be a mapping satisfying the given condition.  Let $y_0\in X$ be an arbitrary point. Define a sequence$<y_n>$ by the relation $y_n=f^ny_0$, then by the given condition we have
\begin{equation*}
\begin{split}
G_r(fy_{n-1},fy_n, \dots, fy_n) \le  k&\,\text{max} \{ G_r(y_{n-1},y_n,\dots, y_n),G_r(y_{n-1},fy_{n-1},\dots ,fy_{n-1}),\\
&\dots \dots , G_r(y_n,fy_n,\dots ,fy_n),G_r(y_{n-1},fy_n,\dots ,fy_n),\\
&\qquad G_r(y_n,fy_n,\dots ,fy_n),\dots ,G_r(y_n,fy_{n-1},\dots ,fy_{n-1})\}
\end{split}
\end{equation*}
which gives
\begin{equation}
\label{equ3}
G_r(y_n,y_{n+1},\dots ,y_{n+1})\le k\,\text{max}\{ G_r(y_{n-1},y_n,\dots ,y_n),G_r(y_{n-1},y_{n+1},\dots ,y_{n+1})\}
\end{equation}
By [G 5] we have
\begin{equation*}
G_r(y_{n-1},y_{n+1},\dots ,y_{n+1})\le G_r(y_{n-1},y_n,\dots ,y_n)+G_r(y_n,y_{n+1},\dots ,y_{n+1})
\end{equation*}
Hence from ~\eqref{equ3} we have
\begin{multline*}
G_r(y_n,y_{n+1},\dots ,y_{n+1})\le k\, \text{max} \{G_r(y_{n-1},y_n,\dots ,y_n),G_r(y_{n-1},y_n,\dots ,y_n)+\\
G_r(y_n,y_{n+1},\dots ,y_{n+1})\}
\end{multline*}
Thus 
\begin{equation*}
G_r(y_n,y_{n+1},\dots ,y_{n+1})\le k\,  \{G_r(y_{n-1},y_n,\dots ,y_n)+G_r(y_n,y_{n+1},\dots ,y_{n+1})\}
\end{equation*}
Which gives
\begin{equation}
\label{equ4}
G_r(y_n,y_{n+1},\dots ,y_{n+1})\le  \frac{k}{1-k}\,G_r(y_{n-1},y_n,\dots ,y_n)
\end{equation}
Let $q=\frac{k}{1-k}$ , then $q<1$ since $0\le k<1/2$ and by repeated application of ~\eqref{equ4} we have
\begin{equation}
\label{equ5}
G_r(y_n,y_{n+1},\dots ,y_{n+1})\le  q^n\,G_r(y_0,y_1,\dots ,y_1)
\end{equation}
For all natural numbers $n$ and $m(>n)$ we have by repeated use of [G 5] and ~\eqref{equ5} that
\begin{equation*}
G_r(y_n,y_m,\dots,y_m)\le \frac{q^n}{1-q}\,G_r(y_0,y_1,\dots ,y_1)\rightarrow 0 \, \text{as} \, n,m \rightarrow \infty
\end{equation*}
Thus the sequence $<y_n>$ is a $G_r$-Cauchy sequence in $X$. By completeness of $(X,G_r)$, there exists a point $u\in X$ such that $<y_n>$ is $G_r$-convergent to $u$.\\
Suppose that $fu\neq u$, then
\begin{equation*}
\begin{split}
G_r(y_n,fu, \dots, fu) \le  k&\,\text{max} \{ G_r(y_{n-1},u,\dots, u),G_r(y_{n-1},y_n,\dots ,y_n),\dots  , G_r(u,fu,\dots ,fu),\\
&\qquad \quad G_r(y_{n-1},fu,\dots ,fu), G_r(u,fu,\dots ,fu),\dots ,G_r(u,y_n,\dots ,y_n)\}
\end{split}
\end{equation*}
or
\begin{multline*}
G_r(y_n,fu, \dots, fu) \le  k\,\text{max} \{ G_r(y_{n-1},u,\dots, u),G_r(y_{n-1},y_n,\dots ,y_n), G_r(u,fu,\dots ,fu),\\
G_r(y_{n-1},fu,\dots ,fu),G_r(u,y_n,\dots ,y_n)\}
\end{multline*}
Taking the limit as $n\rightarrow \infty$, and using the fact that the function $G_r$ is continuous on its variables, we have $G_r(u,fu,\dots ,fu) \le k\,G_r(u,fu,\dots ,fu)$,which is a contradiction, since $0\le k<1/2$. So we have $u=fu$.\\
For uniqueness of $u$, suppose that $v\neq u$ is such that $fv=v$, then we have 
\begin{equation*}
G_r(u,v,\dots ,v)=G_r(fu,fv,\dots ,fv)\le k\, \text{max} \{G_r(u,v,\dots,v),G_r(v,fu,\dots ,fu)\}
\end{equation*}
or
\begin{equation*}
G_r(u,v,\dots ,v)\le k\, \text{max} \{G_r(u,v,\dots ,v),G_r(v,u,\dots ,u)\}
\end{equation*}
So, it must be the case that $G_r(u,v,\dots ,v)\le k\,G_r(v,u,\dots ,u)$.\\
Again by the same argument we find that $G_r(v,u,\dots ,u)\le k\, G_r(u,v,\dots ,v)$.Thus we have $G_r(u,v,\dots ,v)\le k^2\,G_r(u,v,\dots ,v)$. Which implies that $u=v$, since $0\le k<1/2$.\\
Now to prove that $f$ is generalized $r$-continuous at $u$, let $<y_n>$ be any sequence in $X$ such that  it is $G_r$-convergent to $u$, then 
\begin{equation*}
\begin{split}
G_r(fy_n,fu, \dots, fu) \le  k&\,\text{max} \{ G_r(y_n,u,\dots, u),G_r(y_n,fy_n,\dots ,fy_n),\dots  , G_r(u,fu,\dots ,fu),\\
& G_r(y_n,fu,\dots ,fu),\dots , G_r(u,fu,\dots ,fu) ,G_r(u,fy_n,\dots ,fy_n)\}
\end{split}
\end{equation*}
or 
\begin{equation*}
\begin{split}
G_r(fy_n,u, \dots, u) \le  k&\,\text{max} \{ G_r(y_n,u,\dots, u),G_r(y_n,fy_n,\dots ,fy_n) , G_r(u,fy_n,\dots ,fy_n)\}
\end{split}
\end{equation*}
By [G 5] we have 
\begin{equation*}
G_r(y_n,fy_n,\dots ,fy_n)\le G_r(y_n,u,\dots ,u)+G_r(u,fy_n,\dots ,fy_n)
\end{equation*}
Thus we deduce that
\begin{equation*}
G_r(fy_n,u, \dots, u) \le  k\, \{ G_r(y_n,u,\dots, u)+ G_r(u,fy_n,\dots ,fy_n)\}
\end{equation*}
Using proposition ~\ref{inequalty} we have
\begin{equation*}
G_r(fy_n,u, \dots, u) \le  k\, \{ G_r(y_n,u,\dots, u)+(r-1)\, G_r(fy_n,u,\dots ,u)\}
\end{equation*}
or
\begin{equation*}
G_r(fy_n,u, \dots, u) \le  \frac{k}{1-(r-1)k}\,  G_r(y_n,u,\dots, u)
\end{equation*}
Taking the limit as $n\rightarrow \infty$, we see that $G_r(fy_n,u,\dots ,u)\rightarrow 0$ and so by proposition ~\ref{convsq} the sequence $<fy_n>$ is $G_r$-convergent to $u=fu$. Therefore proposition ~\ref{cont} implies that $f$ is generalized $r$-continuous at $u$.

\end{proof}
\section{Acknowledgement}
The author is thankful to the referee for his valuable comments and suggestions on this manuscript.

\end{document}